\documentclass[12pt,reqno]{amsart}
\usepackage{hyperref,cleveref,amssymb,eqnarray,mathrsfs,xcolor,cite}

\theoremstyle{plain}
\newtheorem{theorem}{Theorem}[section]
\newtheorem{proposition}[theorem]{Proposition}
\newtheorem{corollary}[theorem]{Corollary}

\theoremstyle{definition}

\newcommand{\Pelc}[1]{\left(\Sigma X_n\right)_p}

\DeclareSymbolFont{bbold}{U}{bbold}{m}{n}
\DeclareSymbolFontAlphabet{\mathbbold}{bbold}

\DeclareMathOperator{\tr}{tr}

\DeclareMathOperator{\diag}{Diag}
\renewcommand{\le}{\leqslant}
\renewcommand{\ge}{\geqslant}

\begin{document}

\baselineskip 6.8mm

\title[Commutators of positive operators]{Commutators greater than a perturbation of the identity}

\author{R. Drnov\v sek}
\address{Faculty of Mathematics and Physics, University of Ljubljana,
  Jadranska 19, 1000 Ljubljana, Slovenia \ \ \ and \ \ \
  Institute of Mathematics, Physics, and Mechanics, Jadranska 19, 1000 Ljubljana, Slovenia}
\email{roman.drnovsek@fmf.uni-lj.si}

\author{M. Kandi\'c}
\address{Faculty of Mathematics and Physics, University of Ljubljana,
  Jadranska 19, 1000 Ljubljana, Slovenia \ \ \ and \ \ \
  Institute of Mathematics, Physics, and Mechanics, Jadranska 19, 1000 Ljubljana, Slovenia}
\email{marko.kandic@fmf.uni-lj.si}


\date{\today}

\begin{abstract}
Let $a$ and $b$ be elements of an ordered normed algebra $\mathcal A$ with unit $e$. Suppose that the element $a$ is positive and that for some $\varepsilon>0$ there exists an element $x\in \mathcal A$ with $\|x\|\leq \varepsilon$ such that
$$ ab-ba \geq e+x . $$
 If the norm on $\mathcal A$ is monotone, then we show
$$ \|a\|\cdot \|b\|\geq \tfrac{1}{2} \ln \tfrac{1}{\varepsilon} , $$
which can be viewed as an order analog of Popa's quantitative result for commutators of operators on Hilbert spaces.

We also give a relevant example of positive operators $A$ and $B$ on the Hilbert lattice $\ell^2$ such that their commutator $A B - B A$  is greater than an arbitrarily small perturbation of the identity operator.
\end{abstract}

\maketitle

\section{Introduction and Preliminaries}

Let $\mathcal A$ be a normed algebra with unit $e$. Elements of the form $[a, b]:= a b - b a$ are called {\it commutators}.
As far as we know, the first significant result concerning the characterization of commutators of operators is due to Shoda \cite{Shoda} who proved that an $n\times n$ matrix over a field of zero characteristic is a commutator if and only if its trace is zero. Albert and Muckenhoupt \cite{Albert:57} extended Shoda's result to arbitrary fields. The first major contribution to normed algebras was due to Wintner  and Wielandt  (see e.g. \cite{Win47, Wie49, Hal82}). Their results can be considered as the Wintner-Wielandt theorem which states that the unit element $e$ cannot be expressed as the commutator of elements of $\mathcal A$.
By passing to the Calkin algebra, Halmos observed in \cite{Hal63} that the Wintner-Wielandt result immediately implies
that an operator which is of the form $\lambda I + K$ for some nonzero scalar $\lambda$ and a compact operator $K$ is not a commutator. Due to this fact he introduced \cite[Problem 5]{Hal63} whether every operator which is not of the form  $\lambda I + K$ for some nonzero scalar $\lambda$ and a compact operator $K$ is a commutator. For separable Hilbert spaces Brown and Pearcy \cite{BP64} provided a positive solution to Halmos' problem.  The complete characterization of commutators in
the C*-algebra $\mathcal B(\mathcal H)$ of all operators on a Hilbert space ${\mathcal H}$ is again due to Brown and Pearcy \cite{BP65}.


\begin{theorem} \label{BP}
Let $\mathcal H$ be an infinite-dimensional Hilbert space.
An operator $C \in \mathcal B (\mathcal H)$ is a commutator if and only if it is not of the form
$\lambda I + K$ for some nonzero scalar $\lambda$ and some operator $K$ belonging to the unique maximal ideal in $\mathcal B(\mathcal H)$.
\end{theorem}

For related results about commutators on Banach spaces we refer the reader to \cite{Ap72, Ap73, Dos09, DJ10} and references therein.

Motivated by the classical problem which operators are commutators, the authors of the paper \cite{Bracic:10}  initiated the study of positive commutators of positive operators on Banach lattices. The positivity of operators  $A$ and $B$ may lead to some restrictions on $[A, B]$.  If a positive operator $C$ on a Banach lattice can be written as a commutator of positive operators with one of them compact, then $C$ is necessarily quasinilpotent \cite{Drnovsek:11,Gao:14}.

There is yet another approach how one can study commutators. Popa \cite{Po82} gave the following quantitative version of the Wintner-Wielandt result.
He proved the following bound for the product of norms of operators whenever the commutator is close to the identity.

\begin{theorem} \label{Popa}
Let $\mathcal H$ be an infinite-dimensional Hilbert space. Let $A, B \in \mathcal B (\mathcal H)$ be such that
$$ \| [A, B] - I \| \leq \varepsilon $$
for some $\varepsilon > 0$. Then
$$ \| A \| \cdot \| B \| \geq \tfrac{1}{2} \ln \tfrac{1}{\varepsilon} . $$
\end{theorem}

It follows from Theorem \ref{BP}  that $[A, B]$ can be made arbitrarily close to the identity $I$ in the operator norm for $A, B \in \mathcal B (\mathcal H)$.  In fact, we have the following result; see \cite[Proposition 0.2]{Tao19}.

\begin{theorem} \label{Popa_example}
Let $\mathcal H$ be an infinite-dimensional Hilbert space. Then,  for any $\varepsilon \in (0,1)$, there exist
operators $A, B \in \mathcal B (\mathcal H)$ such that
$$ \| [A, B] - I \| \leq \varepsilon $$
and
$$ \| A \| \cdot \| B \| = O(\varepsilon^{-2}) . $$
\end{theorem}

Here we use the asymptotic notation $r = O(s)$ to denote an estimate of the form $|r| \leq  C \cdot s$  for an absolute constant $C$.

Tao \cite{Tao19} improved Theorem \ref{Popa_example} by obtaining a bound closer to that in Theorem \ref{Popa}.

\begin{theorem} \label{Tao}
Let $\mathcal H$ be an infinite-dimensional Hilbert space. Then,  for any $\varepsilon \in (0,1/2)$, there exist
operators $A, B \in \mathcal B (\mathcal H)$ such that
$$ \| [A, B] - I \| \leq \varepsilon $$
and
$$ \| A \| \cdot \| B \| = O( \ln^5 \tfrac{1}{\varepsilon}) . $$
\end{theorem}

In this paper we prove order analogs of Theorems \ref{Popa} and \ref{Popa_example} in the context of unital ordered normed algebras that are a generalization of ordered Banach algebras of operators on Banach lattices.
We also give a relevant example of positive operators $A$ and $B$ on the Hilbert lattice $\ell^2$ such that their commutator $A B - B A$  is greater than an arbitrarily small perturbation of the identity operator.
By a {\it Hilbert lattice} we understand a Banach lattice whose underlying Banach space is a Hilbert space.

We end this section by recalling some definitions about ordered normed algebras.
Let $\mathcal A$ be a real or complex normed algebra with unit $e$. We
call a nonempty subset $\mathcal A^+$  of $\mathcal A$ a {\it cone} if
$\mathcal A^+ + \mathcal A^+ \subseteq \mathcal A^+$, $\lambda \mathcal A^+ \subseteq \mathcal A^+$  for all $\lambda \geq 0$ and
$\mathcal A^+ \cap ( - \mathcal A^+) = \{0\}$.
Any cone  $\mathcal A^+$ on $\mathcal A$ induces a partial ordering on $\mathcal A$ in the
following way:
$$ a \leq b \ \textrm{ if and only if } \ b - a \in \mathcal A^+, \ (a, b \in \mathcal A) . $$
Clearly, $\mathcal A^+ = \{a \in  \mathcal A : a \geq 0\}$, and therefore we call the elements of  $\mathcal A^+$ {\it positive}, and the elements of  $- \mathcal A^+$ {\it negative}.
A cone $\mathcal A^+$ of a normed algebra $\mathcal A$ with unit $e$  is called an {\it algebra cone} if
$\mathcal A^+ \cdot  \mathcal A^+ \subseteq \mathcal A^+$ and $e \in \mathcal A^+$.
In this case $\mathcal A$ is called an {\it ordered normed algebra}.
A cone $\mathcal A^+$ of $\mathcal A$ is {\it normal} if there exists a positive constant $\alpha$ (necessarily at least $1$) such that
$0 \le a \le b$ implies $\|a\| \le \alpha \|b\|$.  If we can take $\alpha = 1$, then we say that the norm on $\mathcal A$ is {\it monotone}. \\

\section{Commutators in ordered normed algebras}

Let $C \geq I$ be any operator on the Hilbert lattice $\ell^2$ which is
is not of the form $\lambda I + K$ for some scalar $\lambda$ and some compact operator $K$.
Then by \Cref{BP} there exist operators $A$ and $B$ in $\mathcal B (\ell^2)$ such that $[A, B] = C \geq I$.
One may ask whether $A$ and $B$ can be also positive operators on the Hilbert lattice $\ell^2$.
The answer to this question is negative, as we have the following theorem that is inspired by Wielandt's proof of the Wintner-Wielandt result (see e.g. \cite{Hal82}).

\begin{theorem}\label{Wielandt-type}
Let $a$ and $b$ be elements of a unital ordered normed algebra $\mathcal A$ with a normal cone. If $[a,b]\geq e$, then none of $a$ and $b$ is neither positive nor negative.
\end{theorem}

\begin{proof}
Suppose first that $a$ is positive. We will prove by induction that for each $n\in \mathbb N$ we have $[a^n,b]\geq n a^{n-1}$.
For $n=1$, the inequality is clear by the hypothesis. Suppose that $[a^n,b]\geq n a^{n-1}$ for some $n\in \mathbb N$.
Then one can write
\begin{align*}
[a^{n+1},b]&=a^{n+1}b-ba^{n+1}=a(a^nb-ba^n)+(ab-ba)a^n\\
& = a[a^n,b]+[a,b]a^n.
\end{align*}
Since the cone $\mathcal A^+$ is an algebra cone, we have
$$ [a^{n+1},b] \geq a\cdot n a^{n-1}+e\cdot a^n=(n+1)a^n \geq 0 , $$
 which concludes the induction step.

The normality of the cone $\mathcal A^+$ yields the existence of a constant $\alpha \ge 1$ such that $0\leq x\leq y$ implies $\|x\|\leq \alpha \|y\|$. Hence, from the inequality $0\leq na^{n-1}\leq [a^n,b]$ we conclude that
$n\|a^{n-1}\|\leq \alpha \|[a^n,b]\|$, so that the submultiplicativity of the norm gives us
$$n \|a^{n-1}\| \leq \alpha \|[a^{n},b]\| = \alpha \|a^{n}b-ba^{n}\|\leq 2\alpha \|a\|\,\|b\|\,\|a^{n-1}\|.$$
If $a^k=0$ for some $k\in \mathbb N$, then the inequality $0 = [a^k,b]\geq ka^{k-1}\geq 0$ would yield $a^{k-1}=0$. Repeating the same argument, we obtain $e=a^0=0$, which is impossible. Hence, $a^n\neq 0$ for any $n\in \mathbb N$, so that for each $n\in \mathbb N$ we obtain
$n\leq 2\alpha\|a\|\,\|b\|$. This contradiction shows that $a$ is not positive.

If $b$ is positive, we first rewrite $[a,b]=[b,-a]$ and apply the first part of the proof to conclude that the assumption $b\geq 0$ is absurd. If $a$ is negative, we rewrite $[a,b]=[-a,-b]$ and use the same argument above. We similarly treat the case when $b$ is negative.
\end{proof}

\begin{corollary}\label{solvabilityInequality}
Let  $\mathcal A$ be a unital ordered normed algebra with a normal algebra cone  $\mathcal A^+$.
Then the inequality $[a,b]\geq e$ is not solvable for $a$ and $b$ in $\mathcal A$ with one of them positive.
\end{corollary}

\Cref{solvabilityInequality} enables us to prove that in the case of a closed and normal algebra cone a commutator of two elements with one of them being positive cannot dominate a perturbation of the unit element whenever the perturbation is too small in norm.

\begin{corollary}\label{perturbationClosed}
\label{delta}
Let  $\mathcal A$ be a unital ordered normed algebra with a normal and closed algebra cone $\mathcal A^+$.
Let $a$ and $b$ be elements of $\mathcal A$ with one of them positive.
Then there exists a positive real number $\delta>0$ such that there is no element $x \in \mathcal A$ with $\|x\|<\delta$ and
$[a,b]\geq e+x$.
\end{corollary}

\begin{proof}
Suppose that the conclusion of the corollary does not hold. Then for each $n\in \mathbb N$ there exists an element $x_n\in \mathcal A$ such that $\|x_n\|<\tfrac 1n$ and
$[a,b]\geq e+x_n$.  Since the cone $\mathcal A^+$ is closed, we obtain $[a,b]\geq e$ which clearly contradicts \Cref{solvabilityInequality}.
\end{proof}

The natural question that arises here is the necessity of closeness condition of the cone in \Cref{delta}. Although at the first glance the answer is not obvious, it turns out that the closeness condition is redundant once we prove the following order analog of Popa's quantitative version (\Cref{Popa}) of the Wintner-Wielandt result.


\begin{theorem}\label{quantitative_norm}
Let $a$ and $b$ be elements of an ordered normed algebra $\mathcal A$ with unit $e$. Suppose that at least one of the elements $a$ and $b$ is positive, and that for some $\varepsilon>0$ there exists an element $x\in \mathcal A$ with $\|x\|\leq \varepsilon$ such that
$$[a,b]\geq e+x.$$
If the cone $\mathcal A^+$ is normal with normality constant $\alpha$, then
$$\|a\|\cdot \|b\|\geq \tfrac{1}{2\alpha} \ln \tfrac{1}{\alpha \varepsilon}.$$
In particular, if the norm on $\mathcal A$ is monotone, then
$$ \|a\|\cdot \|b\|\geq \tfrac{1}{2} \ln \tfrac{1}{\varepsilon} . $$
\end{theorem}

\begin{proof}
Since $[a,b]=[b,-a]$, we may assume that the element $a$ is positive.
Let us first prove by induction that for each $n\in \mathbb N$ we have
\begin{equation}\label{komutatorska_formula}
[a^n,b]\geq n a^{n-1}+a^{n-1}x + a^{n-2}xa+\cdots +xa^{n-1}.
\end{equation}
For $n=1$ the inequality above reads as
$[a,b]\geq e+x$ which holds by the assumption. Suppose that the inequality above holds for some $n\in \mathbb N$. Then we have
\begin{align*}
[a^{n+1},b]&=a^{n+1}b-ba^{n+1}=a(a^nb-ba^n)+(ab-ba)a^n=a[a^n,b]+[a,b]a^n\\
& \geq a\left(n a^{n-1}+a^{n-1}x + a^{n-2}xa + \cdots + xa^{n-1}\right)+(e+x)a^n\\
& = (n+1)a^n+a^nx+a^{n-1}xa +\cdots+axa^{n-1}+xa^n
\end{align*}
which completes the induction step.

If we rewrite inequality \eqref{komutatorska_formula} as
$$0\leq n a^{n-1}\leq [a^n,b]-a^{n-1}x - a^{n-2}xa-\cdots -xa^{n-1}$$
and use the fact that the cone of $\mathcal A$ is normal with $\alpha$ its normality constant, we obtain
$$n\|a^{n-1}\|\leq \alpha (\|[a^n,b]\|+n\varepsilon \|a\|^{n-1}).$$
Since the norm of the commutator $[a^n,b]$ can be estimated as
$$\|[a^n,b]\|\leq \|a^nb\|+\|ba^n\|\leq 2\|a^n\|\|b\| , $$
we obtain
$$n\|a^{n-1}\|\leq 2\alpha\|a^n\|\|b\|+\alpha\varepsilon n\|a\|^{n-1}.$$

Assume first that $\|b\|=\tfrac{1}{2\alpha}$. Dividing by $n!$ and taking sums where $n$ goes from one to infinity we have
$$\sum_{n=1}^\infty \frac{\|a^{n-1}\|}{(n-1)!}\leq \sum_{n=1}^\infty\frac{\|a^n\|}{n!} + \alpha\varepsilon \sum_{n=1}^\infty \frac{\|a\|^{n-1}}{(n-1)!}.$$
Since all terms on the left-hand side of the inequality cancel except the first one, we have
$$1\leq \alpha \varepsilon \sum_{n=1}^\infty \frac{\|a\|^{n-1}}{(n-1)!}=\alpha \varepsilon e^{\|a\|}$$
and so $\|a\|\geq \ln \tfrac{1}{\alpha \varepsilon}$.

For the general case, replace $a$ and $b$ by $\widetilde a:=2\alpha \|b\|a$ and $\widetilde b:=\tfrac{1}{2\alpha\|b\|} b$, respectively. Since
$$[\widetilde a,\widetilde b]=[a,b]\geq e+x,$$ the first part of the proof yields $2\alpha\|a\|\|b\|=\|\widetilde a\|\geq \ln \tfrac{1}{\alpha\varepsilon}$ which completes the proof.
\end{proof}

As it was already announced, the following corollary is a significant improvement of \Cref{perturbationClosed} to all ordered normed algebras with normal algebra cones.

\begin{corollary}
\label{delta1}
Let  $\mathcal A$ be a unital ordered normed algebra with a normal algebra cone $\mathcal A^+$.
Let $a$ and $b$ be elements of $\mathcal A$ with one of them positive.
Then for each
$$0<\delta<\tfrac{1}{\alpha}e^{-2\alpha\|a\|\,\|b\|}$$ there is no element $x \in \mathcal A$ with $\|x\|<\delta$ and
$[a,b]\geq e+x$.
\end{corollary}

\begin{proof}
Suppose there exists a positive real number $\delta<\tfrac{1}{\alpha}e^{-2\alpha\|a\|\,\|b\|}$ such that $[a,b]\geq e+x$
for some element $x \in \mathcal A$ with $\|x\|<\delta$.
Then \Cref{quantitative_norm} yields $$\|a\|\cdot \|b\|\geq \tfrac{1}{2\alpha} \ln \tfrac{1}{\alpha \delta}>\|a\|\cdot\|b\|$$ which is clearly impossible.
\end{proof}

The bound $\delta$ clearly depends on the choice of elements $a$ and $b$ and also on normality constant $\alpha$.
The application of the trace in the case of $n\times n$ matrices $M_n(\mathbb R)$ equipped with partial ordering defined entrywise will give a more precise lower bound for $\delta$. Namely, in \Cref{lower_bound_matrices} we will prove that whenever the inequality $[A,B]\geq I+X$ is solvable in $M_n(\mathbb R)$ then $\|X\| \geq 1$. \\



\section{Commutators (of positive operators) greater than a perturbation of the identity}

In view of \Cref{Popa_example} and \Cref{quantitative_norm} one can pose a question whether there exist positive operators $A$ and $B$ on the Hilbert lattice $\ell^2$ such that their commutator $[A,B]$  is greater than an arbitrarily small perturbation of the identity. In this section we find such operators. It turns out that this task is harder than one can imagine.

Let $(e_n)_{n\in \mathbb N}$ be the standard basis of the Hilbert lattice $\ell^2$. It should be clear that the bounded operator $U\colon \ell^2\to \ell^2$ which is defined on the standard basis vectors
as $Ue_n=e_{2n}$ ($n\in\mathbb N$) is a positive isometry. Similarly, the operator $V\colon \ell^2\to \ell^2$ which is defined on the standard basis vectors as $Ve_n=e_{2n-1}$ ($n\in\mathbb N$) is also a positive isometry.
Hence,  $U$ and $V$ can be realized as infinite matrices
\begin{equation}\label{UVmatrike}
U=\left(
\begin{array}{ccccccc}
0 & 0 & 0 & 0 & 0 & 0 & \hdots\\
1 & 0 & 0 & 0 & 0 & 0 & \hdots\\
0 & 0 & 0 & 0 & 0 & 0 & \hdots\\
0 & 1 & 0 & 0 & 0 & 0 & \hdots\\
0 & 0 & 0 & 0 & 0 & 0 & \hdots\\
0 & 0 & 1 & 0 & 0 & 0 & \hdots\\
\vdots & \vdots & \vdots & \vdots & \vdots & \vdots & \ddots
\end{array}\right) \qquad \textrm{and}\qquad
V=\left(
\begin{array}{ccccccc}
1 & 0 & 0 & 0 & 0 & 0 & \hdots\\
0 & 0 & 0 & 0 & 0 & 0 & \hdots\\
0 & 1 & 0 & 0 & 0 & 0 & \hdots\\
0 & 0 & 0 & 0 & 0 & 0 & \hdots\\
0 & 0 & 1 & 0 & 0 & 0 & \hdots\\
0 & 0 & 0 & 0 & 0 & 0 & \hdots\\
\vdots & \vdots & \vdots & \vdots & \vdots & \vdots & \ddots
\end{array}
\right)
\end{equation}
with respect to the standard basis $(e_n)_{n\in\mathbb N}$.
The operator $X\colon \ell^2\oplus \ell^2\to \ell^2$ is defined as the block operator matrix
$X=\left(\begin{array}{cc}
U & V
\end{array}\right)$.
Having in mind \eqref{UVmatrike}, we conclude that $X$ is a unitary operator. Hence, the identities $X^*X=I_{\ell^2\oplus \ell^2}$ and $XX^*=I_{\ell^2}$ yield
$$I_{\ell^2}=XX^*=\left(\begin{array}{cc}
U & V
\end{array}\right)\left(\begin{array}{c}
U^* \\
V^*
\end{array}\right)=UU^*+VV^*$$
and
$$\left(\begin{array}{cc}
I & 0\\
0 & I
\end{array}\right)=I_{\ell^2\oplus \ell^2}=X^*X=\left(\begin{array}{c}
U^* \\
V^*
\end{array}\right)\left(\begin{array}{cc}
U & V
\end{array}\right)=
\left(\begin{array}{cc}
U^*U & U^*V\\
V^*U & V^*V
\end{array}\right) , $$
that is, we have $U^*U = V^*V = I$ and $U^*V=V^*U=0$.

Let us define the operator $C \colon \ell^2\oplus \ell^2\to \ell^2\oplus \ell^2$ as
$$C=\left(\begin{array}{cc}
I & 0\\
0 & 0\\
\end{array}\right).$$

The following proposition helps us implicitly to find an example we are searching for.
Recall that operators of the form $[T^*, T] = T^* T - T T^*$ are called {\it self-commutators}, where $T$ is an operator on a Hilbert space.

\begin{proposition}
The operator $C$ can be written as a self-commutator of a positive isometry.
\end{proposition}

\begin{proof}
Upon a permutation similarity we may assume that
$C$ is the infinite matrix
$$\left(
\begin{array}{cccccc}
1 & 0 & 0 & 0 & 0 & \hdots\\
0 & 0 & 0 & 0 & 0 & \hdots\\
0 & 0 & 1 & 0 & 0 & \hdots\\
0 & 0 & 0 & 0 & 0 & \hdots\\
0 & 0 & 0 & 0 & 1 & \hdots\\
\vdots & \vdots & \vdots & \vdots & \vdots & \ddots
\end{array}\right)$$
with respect to the standard basis $(e_n)_{n\in \mathbb N}$.
Since the positive isometry $U$ in (\ref{UVmatrike}) satisfies equalities $U^*U=I$ and $UU^*=I-C$, we obtain $[U^*,U]=C$,
completing the proof.
\end{proof}

As announced, we now show the order analog of Theorem \ref{Popa_example}.

\begin{theorem}\label{I+N}
There exist positive operators $A,B\colon \ell^2\to \ell^2$ such that $[A,B]=I+N$, where $N$ is a nilpotent operator of nil-index $3$. Furthermore,
if $\varepsilon \in (0,1)$, then $A$ and $B$ can be chosen in such a way that  $\|A\|=O(\varepsilon{^{-3}})$, $\|B\|=O(\varepsilon^{-3})$ and $\|N\|=O(\varepsilon)$.
\end{theorem}

In view of \Cref{Wielandt-type} it is not possible that the nilpotent operator $N$ in Theorem \ref{I+N} is positive.

\begin{proof}
Let $W:=UV^*+VU^*$. By a direct calculation one can show that $W$ satisfies $We_{2n}=e_{2n-1}$ and $We_{2n-1}=e_{2n}$ for each $n\in \mathbb N$. Therefore,
$$W=
\left(
\begin{array}{ccccc}
0 & 1 & 0 & 0 & \hdots\\
1 & 0 & 0 & 0 & \hdots\\
0 & 0 & 0 & 1 & \hdots\\
0 & 0 & 1 & 0 & \hdots\\
\vdots & \vdots & \vdots & \vdots & \ddots
\end{array}\right).$$
Now we define $4\times 4$ block-operator matrices
$$A=\left(\begin{array}{cccc}
0 & V^* & 0 & 3I\\
0 & U^* & I & 0\\
V^* & 0 & U^* & 2W\\
U^* & 0 & V^* & 0
\end{array}\right) \qquad \textrm{and} \qquad
B=\left(\begin{array}{cccc}
0 & 0 & 2V & 2U\\
0 & 0 & 0 & 0\\
0 & I & 2U & 2V\\
I & 0 & 0 & 0
\end{array}\right) $$
that define positive operators on $\ell^2 \cong \ell^2  \oplus \ell^2 \oplus \ell^2 \oplus \ell^2$.
Using the facts that $U$ and $V$ are isometries, and $U^*V=V^*U=0$ we obtain
$$AB=\left(
\begin{array}{cccc}
3I & 0 & 0 & 0 \\
0 & I & 2U & 2V\\
2W & U^* & 4I & 0\\
0 & V^* & 0 & 4I \\
\end{array}\right).$$
Similarly, applying the identities $W=UV^*+VU^*$ and $I=UU^*+VV^*$ we obtain
$$
BA=\left(
\begin{array}{cccc}
2I & 0 & 2W & 4VW \\
0 & 0 & 0 & 0\\
2W & U^* & 3I & 4UW\\
0 & V^* & 0 & 3I
\end{array}\right).$$
A direct calculation now yields
$$[A,B]=\left(\begin{array}{cccc}
I & 0 & -2W & -4VW\\
0 & I & 2U & 2V\\
0 & 0 & I & -4UW\\
0 & 0 & 0 & I
\end{array}\right).$$
To conclude the first part of the proof, observe that $[A,B]$ can be written in the form
$I+N$, where
$$N=\left(\begin{array}{cccc}
0 & 0 & -2W & -4VW\\
0 & 0 & 2U & 2V\\
0 & 0 & 0 & -4UW\\
0 & 0 & 0 & 0
\end{array}\right)$$
is nilpotent. Its nil-index is $3$, since
$$N^2=\left(\begin{array}{cccc}
0 & 0 & 0 &  8WUW \\
0 & 0 & 0 &  -8U^2 W \\
0 & 0 & 0 & 0 \\
0 & 0 & 0 & 0
\end{array}\right) \neq 0 $$
(as $U$ and $W$ are isometries), and $N^3 = 0$.

For the ``furthermore" statement, let us redefine the operators $A, B$ and $N$ as
$\widetilde A=S_\varepsilon AS_\varepsilon^{-1}$, $\widetilde B=S_\varepsilon BS_\varepsilon^{-1}$ and $\widetilde N=S_\varepsilon N S_\varepsilon^{-1}$, where $S_\varepsilon=\diag(\varepsilon^3 I, \varepsilon^2 I, \varepsilon I, I)$ is the block-diagonal operator matrix. From $[A,B]=I+N$ we conclude
$$ [\widetilde A,\widetilde B] =S_\varepsilon ABS_\varepsilon^{-1}-S_\varepsilon BAS_\varepsilon^{-1}  =
S_\varepsilon[A,B]S_\varepsilon^{-1}=S_\varepsilon(I+N)S_\varepsilon^{-1}=I+\widetilde N. $$
A direct calculation shows
$$\widetilde A=
\left(\begin{array}{cccc}
0 & \varepsilon V^* & 0 & 3\varepsilon^3 I\\
0 & U^* & \varepsilon I & 0\\
\frac{1}{\varepsilon^2}V^* & 0 & U^* & 2\varepsilon W\\
\frac{1}{\varepsilon^3}U^* & 0 & \frac{1}{\varepsilon} V^* & 0
\end{array}\right)$$
which yields $\|\widetilde A\|=O(\varepsilon^{-3})$. Similarly, one can show $\|\widetilde B\|=O(\varepsilon^{-3})$.
To finish the proof, note that the equality
$$\widetilde N=\left(\begin{array}{cccc}
0 & 0 & -2\varepsilon^2 W & -4\varepsilon^3 VW\\
0 & 0 & 2\varepsilon U & 2\varepsilon^2 V \\
0 & 0 & 0 & -4\varepsilon UW \\
0 & 0 & 0 & 0
\end{array}\right)$$
yields $\|\widetilde N\|=O(\varepsilon)$.
\end{proof}

\section{The finite-dimensional case}

In this section we study the above results in the finite-dimensional setting. The situation here is quite different, as the notion of the trace restricts drastically which operators are commutators.
By \cite{Shoda, Albert:57}, an $n\times n$ matrix over an arbitrary field is a commutator if and only if its trace is zero.
In \cite[Corollary 3.2]{DK19} the authors proved the following order analog of this result which characterizes positive matrices as commutators of positive matrices.

\begin{proposition}
A positive matrix $C$ can be written as a commutator of positive matrices $A$ and $B$ if and only if $C$ is nilpotent.
Moreover, we can choose $A$ to be diagonal and $B$ to be permutation similar to a strictly upper-triangular matrix.
\label{positive_matrices}
\end{proposition}

Motivated with \cite{JOS}, we now prove the following quantitative version of \Cref{positive_matrices}.

\begin{proposition}
Let $C$ be a positive nilpotent matrix. Then for each $\varepsilon > 0$ there exist a positive diagonal matrix $A$ and a positive matrix $B$  such that $C = AB - BA$ and $B A \leq \varepsilon C$.
\end{proposition}

\begin{proof}
We may assume that $C = (c_{i j})_{i,j=1}^n$ is a strictly lower-triangular matrix.
Define $A$ to be the diagonal matrix with diagonal entries
$a_{k k} = \left( \frac{1 + \varepsilon}{\varepsilon} \right)^{k-1}$ for $k=1, 2, 3, \ldots, n$.
The positive matrix $B$ is defined as
$$b_{ij}:=\left\{
\begin{array}{ccc}
\frac{c_{ij}}{a_ {i i}-a_{j j}} &:& i > j\\
0 & : & i \leq j
\end{array}.\right.$$
A direct calculation shows that $C=AB-BA$. To prove the inequality $B A \leq \varepsilon C$, we use the following estimate
$$ \frac{a_{j j}}{a_{i i} - a_{j j}} \leq \frac{a_{j j}}{a_{j+1,j+1} - a_{j j}} = \frac{1}{\frac{1 + \varepsilon}{\varepsilon} - 1} = \varepsilon $$
that holds for $i > j$.
\end{proof}

We continue this section with the following result that characterizes which positive matrices are commutators of two matrices with one them being positive.

\begin{proposition}
A positive $n \times n$ matrix $C$ is a commutator of $A$ and $B$ with $A\geq 0$ if and only if $\tr(C)=0$.
\end{proposition}

\begin{proof}
Since every commutator has trace zero, it only suffices to prove the ``if" statement.
To this end, note first that $C\geq 0$ and $0=\tr (C)=\sum_{k=1}^n c_{kk}$ yield $c_{11}=\cdots=c_{nn}=0$.
Define $A$ to be the diagonal matrix with diagonal entries $a_{kk}=k$ for each $k=1,\ldots,n$. In particular, $A$ is a positive matrix. The matrix $B$ is defined as
$$b_{ij}:=\left\{
\begin{array}{ccc}
\frac{c_{ij}}{i-j} &:& i\neq j\\
0 & : & i=j
\end{array}.\right.$$
A direct calculation shows
that $C=AB-BA$.
\end{proof}

We conclude this paper with the following quantitative version of Popa's result for $n\times n$ matrices equipped with the partial ordering defined entrywise.
By $r(T) $ we denote the spectral radius of a matrix $T$.

\begin{proposition}\label{lower_bound_matrices}
Let $A$ and $B$ be real $n\times n$ matrices. If there exists a real $n\times n$ matrix $X$ such that
$$[A,B]\geq I-X,$$ then the following assertions hold:
\begin{enumerate}
\item $\tr(X)\geq n$;
\item $r(X)\geq 1$, and so $\| X \| \geq 1$;
\item if $X$ is an idempotent, then $X=I$.
\end{enumerate}
\end{proposition}

\begin{proof}
(i) Applying the trace on both sides of the inequality, we obtain
$$ 0=\tr([A,B])\geq \tr(I-X)=n-\tr (X) $$
which yields $\tr(X)\geq n$.

(ii) Suppose that $r(X)<1$. Since $\tr (X)$ is equal to the sum of all eigenvalues $\lambda_1,\ldots,\lambda_n$ of $X$ counted according to their algebraic multiplicities, we have
$$|\tr(X)|=|\lambda_1+\cdots+\lambda_n|\leq |\lambda_1|+\cdots+|\lambda_n| \leq n r(X)<n.$$
This gives us $\tr(X)<n$ which contradicts (i).

(iii) If $X$ is an idempotent, then its eigenvalues are $0$ and $1$. If $X\neq I$, then at least one eigenvalue of $X$ is zero. If $k<n$ is the algebraic multiplicity of the eigenvalue $1$, we have
$\tr(X)=k<n$. Again, this contradicts (i). Hence, $X=I$.
\end{proof}

It should be noted that \Cref{I+N} shows that there is no reasonable infinite-dimensional version of \Cref{lower_bound_matrices}.

\subsection*{Acknowledgements}
The first author is supported by the Slovenian Research and Innovation Agency program P1-0222.
The second author is supported by the Slovenian Research and Innovation Agency program P1-0222 and grant N1-0217.



\end{document}